\documentclass[11pt, a4paper]{article}

\usepackage{upref,amsmath,amssymb,amsthm}
\usepackage{amsmath}
\usepackage{graphicx}
\usepackage[swedish,english]{babel}

\usepackage{amsmath,amssymb,upref,amsthm,eucal}
\usepackage{graphics}
\usepackage[swedish,english]{babel}
\newtheorem{theorem}{Theorem}

\newtheorem{remark}[theorem]{Remark}
\newtheorem{proposition}[theorem]{Proposition}
\newtheorem{definition}[theorem]{Definition}
\theoremstyle{remark}

\begin{document}

\title{A Functional Hodrick-Prescott Filter}

\author{Boualem Djehiche\thanks{Department of Mathematics, The Royal Institute of
 Technology, S-100 44 Stockholm, Sweden. e-mail: boualem@math.kth.se}\, \thanks{Financial support from the Swedish Export Corporation (SEK) is gratefully acknowledge
 }\,\, and \,  Hiba Nassar\thanks{School of Computer Science, Physics and Mathematics, Linnaeus University, Vejdesplats 7, SE-351 95 V\"axj\"o, Sweden. e-mail: hiba.nassar@lnu.se}}

\date{ November 14, 2013}
 \maketitle
 
\begin{abstract}
We propose a functional version of the Hodrick-Prescott filter for functional data which take values in an infinite dimensional separable Hilbert space. We further characterize the associated optimal smoothing parameter when the associated linear operator is compact and the underlying distribution of the data is Gaussian.

\end{abstract}
\medskip



\section {Introduction}
Functional data analysis is attracting a lot of interest within the fields of statistical estimation and statistical inverse problems  (see Ramsay and Silverman (1997), Bosq (2000),  M\"uller and Stadtm\"uller (2005), Ferraty and Vieu (2006), Stuart (2010)  and Agapiou, Larsson and Stuart (2012) among many other contributions). This is due to the ability of modern instruments to perform tightly spaced measurements so that these data can be seen as samples of curves,  making classical data smoothing techniques and regression models, such as generalized linear models, inadequate tools. Therefore, the need of extending statistical methods from multivariate data to functional data remains an important task for modern statistical theory. Fields of application where functional data are by now natural objects of study include chemometrics (see e.g. Frank and Friedman (1993) and Hastie and Mallows (1993)), climatology (Besse, Cardot and Stephenson (2000)), finance (Preda and Saporta (2005)) and  linguistic (Hastie {\it et al.} (1995)), to mention only few. 

\medskip\noindent
 The classical Hodrick-Prescott filter (called henceforth the HP filter)  is  widely used in graduation or data smoothing, and remains a powerful tool to construct life tables in actuarial science and identify trends or business growth, structural breaks and anomalies in econometric data series. The univariate HP filter extracts a 'signal' (also called trend in the economic literature) $y(\alpha,x)=(y_1(\alpha,x),\ldots, y_T(\alpha,x))$ from a noisy time series $x=(x_1,\ldots,x_T)$ as a minimizer of
\begin{equation}\label{HP}
\sum_{t=1}^T(x_t-y_t)^2+\alpha\sum_{t=3}^{T}( y_{t}-2 y_{t-1}+ y_{t-2})^2,
\end{equation}
with respect to $y=(y_1,\ldots,y_T)$, for an appropriately chosen positive parameter $\alpha$, called the smoothing parameter. To determine an appropriate value of the smoothing parameter $\alpha$, Hodrick and Prescott (1997) suggest the time series $(x,y)$  satisfies the following linear mixed model:
\begin{equation}\label{M-HP}
\left\{\begin{array}{lll}
x=y+u,\\
Py=v.
\end{array}\right.
\end{equation}
where, $u\sim N(0,\sigma_u^2I_T)$ and $v\sim N(0,\sigma_v^2I_{T-2})$ ($I_T$ and $I_{T-2}$ denote the $ T \times T$ and $(T-2) \times (T-2)$ identity matrices, respectively) and $P$ is the second order differencing operator $(Py)(t):=y_{t}-2 y_{t-1}+ y_{t-2}, \, t=3,\ldots,T$. Using this model, the appropriate smoothing parameter turns out to be the so-called 'noise-to-signal ratio' $\alpha^*=\sigma_u^2/\sigma_v^2$ (see Schlicht (2005)). This parameter satisfies
\begin{equation}\label{opt-alpha-bis}
E[\, y|\,x]= y(\frac{\sigma_u^2}{\sigma_v^2},x)
\end{equation}
and is optimal in the sense that (see Dermoune {\it et al.} (2009)) the smoothing parameter  minimizes the mean square difference between the 'optimal signal' $y(\alpha,x)$ and the conditional expectation $E[\, y|\,x]$ which is the best predictor (in the $L^2$-norm) of any signal $y$ given the time series $x$. More precisely,
 \begin{equation}\label{opt-alpha}
\sigma_u^2/\sigma_v^2=\arg\min_{\alpha}\left\{\|E[\, y|\,x]-y(\alpha,x)\|^2\right\}.
\end{equation}
Furthermore,  Dermoune {\it et al.} (2009) proposed a multivariate version of the HP filter and determined the possible optimal smoothing parameters.

\medskip\noindent  In this paper, we propose a functional version of the HP filter to extract a 'smooth signal' $y$ from observations $x$ which take values in some function space and are corrupted by a noise $u$ which is apriori unobservable, where the smoothness of the signal is measured by the action of a given linear operator $A$ on $y$, possibly corrupted by a noise $v$  which is also apriori unobservable:
\begin{equation}\label{IM-HP}
\left\{\begin{array}{lll}
x=y+u,\\
Ay=v,
\end{array}\right.
\end{equation}
The filter is of the same form as (\ref{M-HP}), where the second order differencing operator $P$ is replaced by a linear operator $A$. Optimality of such a signal is defined through 

\begin{equation*}\label{HP-trend}
y(\alpha,x) := \arg \min_y \left\{ \left\| x-y \right\|_{H_1}^2 + \alpha \left\|Ay\right\|_{H_2}^2 \right\}, 
\end{equation*}
for appropriate spaces $H_1$ and $H_2$, for a given 'smoothing parameter' $\alpha>0$. 

\medskip\noindent In this paper we consider the case where $H_1$ and $H_2$ are separable Hilbert spaces, and $A: H_1\longrightarrow H_2$ is a compact linear operator. Examples include the following cases of particular interest in quantum mechanics among many other cases (see \cite{Do} for further details): 

\medskip\begin{enumerate}
\item[a.] $H_1=H_2:=L^2(\Omega)$, where, $\Omega$ is a bounded domain in $\mathbb{R}^d$, and $A:=\left(-\Delta\right)^{-\gamma}$, for some $\gamma> d/2$, which is a trace-class operator.

\item[b.] $H_1=H_2:=L^2(\Omega)$, where, $\Omega$ is a domain in  $\mathbb{R}^d$, and $V\in L_{loc}^1(\Omega)$ a potential bounded from below and is such that  $(V)^{\frac{d}{2}-\gamma}\in L^1(\Omega)$, for some $\gamma> d/2$. $A:=\left(-\Delta+V\right)^{-\gamma}$  which is a trace-class operator.
\item [c.] $H_1=H_2:=L^2(\Omega)$, where, $\Omega$ is a domain in $\mathbb{R}^d$, and $V$ is a potential for which the Sch\"oridinger operator $-\Delta+V$ has  eigenvalues $\lambda_1<\lambda_2\le \lambda_3\le\cdots$ diverging to infinity such that $\sum_{n}F(\lambda_n)$ is finite, where $F: \mathbb{R}\longrightarrow \mathbb{R}\cup \{\ +\infty\}$ is convex. We may consider the operator $A:=F(-\Delta+V)$ which is trace-class.
\end{enumerate} 

\medskip\noindent  Given two separable Hilbert spaces $H_1$ and $H_2$ and  a compact linear operator $A: H_1\longrightarrow H_2$,  the main result of this paper is a characterization of the optimal smoothing parameter determined by a criterion similar to (\ref{opt-alpha}), when $u$ and  $v$ are independent Hilbert space-valued Gaussian random variables with zero means and covariance operators $\Sigma_u$ and $\Sigma_v$, i.e.,  self-adjoint, positive, semi-definite and bounded operators in their respective Hilbert spaces. Statistical estimation of the smoothing parameter, given the functional data $x$, will be studied in a future work.

\medskip\noindent
The paper is organized as follows. In Section 2, we introduce a functional version of HP filter under the assumption that the operator $A$ is compact. In Section 3, we characterize the optimal smoothing parameter when the covariance operators $\Sigma_u$  and $\Sigma_v $ are trace class operators. In Section 4,  we extend this characterization to the case where the covariance operators $\Sigma_u$  and $\Sigma_v $ are not trace class, which includes white noise, a favorite model for the processes $u$ and $v$ in the HP filter literature (see Schlicht (2005); Dermoune {\it et al.} (2009) and the references therein). 

\medskip\noindent In a subsequent work we will consider the more general case where the operator $A$ is not necessarily compact.

\section{A Hilbert space-valued Hodrick-Prescott filter} 
In this section, we extend the Hodrick-Prescott filter to an infinite dimensional Hilbert space setting.

\medskip\noindent 
Let $H_1$ and $H_2$ be two separable Hilbert spaces, with norms $\left\| \cdot\, \right\|_{H_i}$ and inner products $\langle \cdot ,\cdot \rangle _{H_i},\,\, i=1,2$, and $x \in H_1$ be a functional time series of observables. As a natural extension of the finite-dimensional case, a functional version of the Hodrick-Prescott filter reconstructs an 'optimal smooth signal' $y\in  H_1$ that solves an equation $Ay=v$, corrupted by a noise $v$ which is apriori unobservable, from observations $x$ corrupted by a noise $u$ which is also apriori unobservable: 
\begin{equation}\label{hod}
\left\{\begin {split}
&x= y+u, \\
&Ay=v,
\end{split}
\right.
\end{equation}
given the linear operator  $A: H_1 \longrightarrow H_2$.

\noindent Optimality of such a smooth signal is defined through a Tikhonov-Phillips regularization of the system (\ref{hod}), by introducing a smoothing parameter $\alpha > 0 $ and minimizing 
\begin{equation}
 \left\| x-y \right\|_{H_1}^2 +  \alpha \left\|Ay\right\|_{H_2}^2
 \label{oreg}
 \end{equation}
with respect to $y$, i.e. an 'optimal smooth' signal associated with $x$ is
\begin{equation}\label{HP-trend}
y(\alpha,x) := \arg \min_y \left\{ \left\| x-y \right\|_{H_1}^2 + \alpha \left\|Ay\right\|_{H_2}^2 \right\}. 
\end{equation}
We refer to the nice survey paper by Stuart (2010) on the Bayesian perspective of this type of inverse problems.

\medskip\noindent As noted in  Dermoune {\it et al.} (2009), to get a feasible solution of the problem (\ref{HP-trend}), it is often necessary to regularize the system (\ref{hod}) with a linear operator $B: H_2\longrightarrow H_2$ with suitable properties instead of a smoothing parameter $\alpha > 0$.  In this case, our 'optimal smooth' signal associated with $x$ reads
\begin{equation}\label{HP-H-trend}
y(B,x) := \arg \min_y \left\{ \left\| x-y \right\|_{H_1}^2 +  \langle Ay, BAy\rangle_{H_2} \right\},
\end{equation} 
provided that 
$$
\langle Ah, BAh\rangle_{H_2} \ge 0,\quad h\in H_1.
$$
Manifestly, setting $B=\alpha I_{H_2}$ in (\ref{HP-H-trend}), we get (\ref{HP-trend}).  

\medskip\noindent Next, we extend the selection criterion (\ref{opt-alpha}) of the optimal smoothing parameter to the infinite dimensional setting.

\begin{definition}\label{best}  The optimal smoothing operator associated with the Hodrick-Prescott filter (\ref{hod}) is the minimizer of the difference between the optimal solution $y(B ,x)$, and the conditional expectation $E[y|x]$, the best predictor of any signal $y$ given the functional data $x$:
\begin{equation}\label{B-best}
\hat B = \arg \min_B \left\| E[y|x] - y(B ,x) \right\|_{H_1}^2.
\end{equation}
\end{definition}

\medskip\noindent  This selection criterion is only useful if it is possible to compute an explicit formula for the conditional expectation, which is not always the case. In the finite dimensional setting, joint Gaussian distributions are among the few probability distributions for which the conditional expectation can be explicitly specified. To extend this criterion to the infinite dimensional setting, we should be able to compute the conditional expectation of Hilbert space-valued (jointly Gaussian) random variables. But, it is well known (see e.g. Rozanov (1968) or Skorohod (1974)) that a Gaussian distribution on a Hilbert space are 'meaningful' if and only if the 'mean' is an element of the underlying Hilbert space and the covariance operator is trace class (thus compact). When the covariance operator is not trace class on the underlying Hilbert space, following e.g. Rozanov (1968) or Lehtinen {\it et al.} (1984), the conditional expectation is again well defined only on a larger Hilbert space to which the covariance operator can be continuously extended to a trace class operator (see Section 4 below).

\medskip\noindent In the next section, we give an explicit form of the optimal smoothing operator $\hat B$, assuming $u$ and $v$ independent and Gaussian, with trace class covariance operators, which insures that the vector $(x,y)$ is a.s. a Hilbert-valued Gaussian vector and, relying on Mandelbaum's results (see Mandelbaum (1984)), we derive an explicit form of the conditional expectation $E[y|x]$.

\section{HP filter associated with trace class covariance operators}
\label{HP filter associated with compact operators}
In this section we study the problem (\ref{HP-H-trend}), for a class of smoothing operators $B$ to be specified below. Moreover, we characterize the optimal smoothing operator which solves (\ref{B-best}), when the covariance operators $\Sigma_u$ and $\Sigma_v$ are trace class. We refer to the classic Reed and Simon (\cite{reed-simon}) for an introduction to compact operators.

\medskip\noindent
Given the linear compact operator $A$, the spaces $H_1$ and $H_2$ admit the following orthogonal decompositions 
\[\begin{split}
H_1 &= \mbox{Ker} (A) \oplus (\mbox{Ker} (A))^\bot,\\
H_2 &= \overline{\mbox{Ran} (A)  }\oplus (\mbox{Ran} (A))^\bot. 
\end{split}
\]
The linear operator $A$ being compact, it admits positive  eigenvalues $\lambda_j,\, j\ge 0$ such that $\lambda_n \searrow 0_+$, as $n\to\infty$. Let $\{\lambda_j^2\} $ be the eigenvalues of $A^*A$ (which is also compact), and $\{e_j \} \in H_1$ its corresponding set of orthonormal eigenvectors.  Since $A^*A$ is self-adjoint on $H_1$,  an orthonormal eigen-basis of $AA^*$ on $H_2$ can be given by
\[
d_j= \frac{1}{\lambda_j} Ae_j, 
\]
 and
\[
\langle d_j,d_l\rangle =\frac{1}{\lambda_j \lambda_l}\langle e_j,A^*Ae_l\rangle =\left\{\begin{array}{lll} 1,\quad j=l,\\ 0,\quad j\neq l.\end{array}\right.
\]
The system $(\lambda_n, e_n, d_n)$ is called the singular value decomposition (SVD) of $A$, and $A$ satisfies
\begin{equation}\label{A}
Ah= \sum _{j=1}^\infty {\lambda_j \langle h , e_j\rangle d_j}, \qquad h\in H_1, 
\end{equation}
where, the sum converges in the operator norm.	

\medskip\noindent
Therefore, the equation $Ay=v $ has a solution of the form 
\begin{equation}\label{solution1}
 y= y_0 + A^* (AA^*)^{-1} v =y_0+ \sum _{j=1} ^\infty {\frac{1}{\lambda_j } \langle v, d_j \rangle e_j },
\end{equation}  where $y_0 \in \mbox{Ker} (A) $ can be chosen arbitrarily, and 
\begin{equation}\label{solution2}
 x= y_0 + A^*(AA^*)^{-1} v + u.
\end{equation}

\medskip\noindent A stochastic model for $(x,y)$ is manifestly determined by models for $y_0$ and $(u,v)$. We assume 

\medskip{\bf Assumption 1.} $y_0$ deterministic.

\medskip{\bf Assumption 2.}  $u$ and $v$ are independent random variables with zero mean and covariance operators $\Sigma_u$ and $\Sigma_v$ respectively.   

\medskip\noindent Assumption 1 is made to ease the analysis. Assumption 2 is natural because apriori there should not be any dependence between the 'residual' $u$ which is due to the noisy observation $x$ and the required degree of smoothness of the signal $y$. 

\medskip Given Assumptions 1 and 2, in view of (\ref{solution1}) and (\ref{solution2}), it holds that $(x,y)$ has 
 mean $(E[x],E[y])=(y_0, y_0)$, and covariance operator
\begin{equation}\label{cov(x,y)}
\begin{split}
\Sigma &= \left( \begin{array}{ccc}
\Sigma _u + Q_v  && Q_v \\ 
Q_v &&  Q_v 
\end{array} 
\right), 
\end{split}
\end{equation}
where,
\begin{equation}\label{Qv}
Q_v:= A^*(AA^*)^{-1}\Sigma_v(AA^*)^{-1}A.
\end{equation}

\medskip\noindent Let $\Pi:= A^*(AA^*)^{-1}A$ denote the orthogonal projector associated with $A$, i.e. $\Pi: H_1\longrightarrow H_1$ is a linear operator, self-adjoint and satisfies $\Pi^2=\Pi$. It is easily checked that, 
for every $\xi\in H_1$,  the elements $\Pi\xi$ and $(I_{H_1}-\Pi)\xi$ are orthogonal:
\begin{equation}\label{Pi-2}
<\Pi\xi,(I_{H_1}-\Pi)\xi >=0,
\end{equation}
and
\begin{equation}\label{Pi}
\mbox{Ker}(A)=\mbox{Ker}(\Pi)=\mbox{Ran}(I_{H_1}-\Pi).
\end{equation}
Moreover, we have $(I_{H_1}-\Pi)y=y_0$ and $A^*(AA^*)^{-1}v=\Pi y$.

\medskip\noindent  In the next proposition we give an explicit expression of the minimizer $y(B,x)$ in (\ref{HP-H-trend}), for a given smoothing operator $B$.

\begin{proposition}\label{mainprop}
Let $ A:H_1 \longrightarrow H_2$ be a compact operator with the singular system $(\lambda_n,e_n, d_n)$. Assume further that the smoothing operator $B: H_2 \longrightarrow H_2$ is linear, bounded and satisfies
\begin{equation}\label{B}
\langle Ah, BAh\rangle_{H_2} \ge 0,\quad h\in H_1.
\end{equation}
Then, there exists a unique $y(B ,x) \in H_1$ which minimizes the functional 
\[  J_B (y) = \left\| x-y \right\|^2_{H_1} + \langle Ay, BAy\rangle_{H_2}.\]
This minimizer is given by the formula 
\begin{equation}\label{mini}
 y(B, x) = (I_{H_1}+ A^*BA)^{-1} x. 
\end{equation}
Moreover, if the smoothing operator $B: H_2 \rightarrow H_2$ admits the following representation
\begin{equation}
Bh= \sum _{k=1}^\infty {\beta_k\langle h, d_k\rangle d_k},\quad h\in H_2,
\end{equation}
\medskip
where $\beta_k>0,\,\, k=1,2,\ldots$, and the sum converges in the operator norm, i.e. $B$ is linear,  compact and injective, then
\begin{equation}\label{mini-2}
y(B, x) = (I_{H_1}+ A^*BA)^{-1} x = \sum _{j=1}^\infty { \frac{1}{1+\lambda_j^2 \beta_j}\langle x ,e_j \rangle e_j}. 
\end{equation}

\end{proposition}
\noindent The proof of this proposition is by now standard (see e.g. Kaipio and Somersalo (2004)), but, we give it here for convenience.

\begin{proof}
In view of (\ref{B}), we have 
\[ \langle x, (I_{H_1}+ A^*BA)x\rangle  = \langle x,x\rangle +\langle x, A^*BA x\rangle = \langle x,x\rangle +\langle Ax,BAx\rangle \geq \left\| x\right\|^2, \]
i.e. the operator $(I+A^*BA)$ is bounded from below. Also 
\[ 
\begin{split}
 | \langle x, (I_{H_1}+A^*BA)y\rangle | &= | \langle x,y\rangle +\langle x, A^*BA y\rangle |\\
 & \leq | \langle x,y\rangle |+  |\langle Ax,BAy\rangle | \\
 & \leq  \left\| x \right\|_{H_1} \left\| y\right\|_{H_1}+ \left\| A\right\|^2\left\|B\right\| \left\| x\right\|_{H_1} \left\|y\right\|_{H_1} \\
 & \leq (1+ \left\| A\right\|^2\left\|B\right\| )  \left\| x \right\|_{H_1}\left\| y\right\|_{H_1},
\end{split}
\]
i.e. the operator $(I_{H_1}+A^*BA)$ is bounded from above. It follows from Riesz' Representation Theorem that the inverse of the operator exists and 
\[ \left\| (I_{H_1} +A^*BA)^{-1} \right\| \leq 1. \]
Hence, $y(B,x)$ in (\ref{mini}) is well defined.
To show that $y(B,x) $ minimizes the functional $J_B$, let $y \in H_1$  arbitrarily chosen, and show that $J_B (y) \geq J_B (y(B,x))$. By decomposing $y$ as
\[ y=y(B,x) + z,\] 
and inserting for the norms of the Hilbert spaces $H_1$ and $H_2$, direct calculations yield
\[ \begin{split}
J_B (y(B,x) + z) &= \left\| x-(y(B,x) +z) \right\|^2_{H_1} + \langle A(y(B,x)+z),BA(y(B,x)+z)\rangle_{H_2}\\
&= J_B (y(B,x)) + \langle z, (I_{H_1}+ A^*BA) z \rangle \\ &+2 \langle z,-x +(I+ A^*BA) y(B,x) \rangle\\
&= J_B (y(B,x)) + \langle z, (I_{H_1}+A^*BA) z \rangle.
\end{split}
\]
The last term is nonnegative and vanishes only if $z=0$. This proves our claim. 

\medskip\noindent In terms of the singular value decomposition of the operators $A$ and $B$, we have
\[ \sum _{j=1} ^\infty { (1+ \lambda _j^2\beta_j)  \langle y(B , x) , e_j \rangle e_j} =\sum _{j=1} ^\infty { \langle x, e_j\rangle e_j}.\]          
By projecting onto the eigenspace $\mbox {Span} \{e_j\} $, we find that 
\[(1+\lambda _j^2\beta_j)  \langle y(B, x) , e_j \rangle = \langle x, e_j\rangle,\]
 i.e. \[ \langle y(B, x) , e_j \rangle = \frac{1}{1+\lambda_j ^2 \beta_j}  \langle x, e_j\rangle.\]
 Hence,
 \begin{equation*}
y(B, x) = (I_{H_1}+ A^*BA)^{-1} x = \sum _{j=1}^\infty { \frac{1}{1+\lambda_j^2 \beta_j}\langle x ,e_j \rangle e_j}. 
\end{equation*}
\end{proof}

\medskip\noindent  We will next compute the conditional expectation of the signal $y$ given the functional data $x$ relaying on the following explicit form of the conditional expectation for jointly Gaussian random variables $X, Y$ with values in a separable Hilbert space $H$, due to Mandelbaum \cite{Mandelbaum}, we recall in the next proposition.

\begin{proposition}\label{Mandelbaum} $($Mandelbaum$)$ Let $X,Y$ be jointly Gaussian $H$-valued random variables. Assume that both $X$ and $Y$ have means $\mu_X$ and $\mu_Y$,  and that the covariance of $X$, $\Sigma_X$, is injective. Then, the conditional expectation of $Y$ given $X$ is 
\begin{equation}
E[Y|X] = \mu_Y+\Sigma _{XY} \Sigma _X^{-1}(X-\mu_X) ,
\label{conditional}
\end {equation}
provided that the operator
 \begin {equation}
 T= \Sigma _{XY} \Sigma _X^{-\frac{1}{2}}
\label{schmidt}
\end{equation}
is Hilbert-Schmidt.
\end{proposition}

\medskip\noindent Using this proposition we will now characterize the optimal smoothing operator, defined by (\ref{B-best}), associated with the Hodrick-Prescott filter (\ref{HP-H-trend}), under the following Assumptions.

\medskip{\bf Assumption 3.} The independent random variables $u$ and $v$ are respectively $N(0,\Sigma_{u})$ and $N(0,\Sigma_{v})$ distributed, 
 where the covariance operators $\Sigma_u$ and $\Sigma_v$ are positive-definite and trace class operators on $H_1$ and $H_2$ respectively.

\medskip\noindent This assumption implies that $\Pi y=A^*(AA^*)^{-1}v$ and $u$ are also independent. Thus, with regard to the following decomposition of $x$,
$$
x=y_0+\Pi y+\Pi u+(I_{H_1}-\Pi)u,
$$
it is natural to assume that even the orthogonal random variables $\Pi u$ and $(I_{H_1}-\Pi)u$ independent. This would mean that the input $x$ is decomposed into three independent random variables. This is actually the case for the classical HP filter. Also, as we will show below, thanks to this property the optimal smoothing operator has the form of a 'noise to signal ratio' in line with the classical HP filter.

\medskip {\bf Assumption 4.} The orthogonal (in $ H_1$) random variables $\Pi u$ and $(I_{H_1}-\Pi)u$ are independent:
\begin{equation}\label{indep-1}
\Pi\Sigma_u=\Sigma_u\Pi.
\end{equation}

\medskip\noindent We note that (\ref{indep-1}) is equivalent to 
\begin{equation}\label{indep-2}
\Pi\Sigma_u\Pi=\Pi\Sigma_u.
\end{equation}

\medskip\noindent Since the covariance operators $\Sigma_u$ and $\Sigma_v$ are trace class and thus compact, by Riesz' Representation Theorem, there exist uniquely determined $ \mu_k>0,\, k=1,2,\ldots$, such that
\begin{equation}\label{sigma-u}
 \Sigma_uh = \sum_{k=1} ^\infty {\mu_k \langle h,e_k\rangle e_k},\qquad h\in H_1,
 \end{equation}
where, the sum converges in the operator norm.
Similarly for  the covariance operator of $v $, there exists uniquely determined $ \tau_k>0,\, k=1,2,\ldots$, such that
\begin{equation}\label{sigma-v}
\Sigma_vh = \sum_{k=1}^\infty {\tau_k \langle h,d_k\rangle d_k},\qquad h\in H_2,
 \end{equation}
where, the sum converges in the operator norm.

\medskip\noindent Hence, in view of Assumptions 1 and 3, and the relations  (\ref{solution1}) and (\ref{solution2}),  $(x,y)$  has joint Gaussian distribution with  mean $(E[x],E[y])=(y_0, y_0)$, and covariance operator
\[ 
\begin{split}
\Sigma &= \left( \begin{array}{ccc}
\Sigma _u + Q_v  && Q_v \\ 
Q_v &&  Q_v 
\end{array} 
\right), 
\end{split}
\]
provided that

\medskip\noindent{\bf Assumption 5.} The operator
\[Q_v:= A^*(AA^*)^{-1}\Sigma_v(AA^*)^{-1}A\]  is trace class.

\medskip\noindent
In terms of the singular value decompositions of $A$, $\Sigma_v$ and $\Sigma_v$ we have
\[ Q_v = \sum_ {k=1} ^\infty {\frac{\tau_k}{\lambda_k^2}\langle . ,e_k\rangle e_k}.\]

\medskip\noindent Therefore, by Assumption 3, 
\begin{equation}
\Sigma _{XY}= Q_v = \sum_{k=1}^\infty {\frac{\tau_k}{\lambda_k^2} \langle .,e_k\rangle e_k}, 
\label{sigmaxy}
\end{equation}
and 
\begin{equation}
\Sigma _{X} = \Sigma_u + Q_v = \sum_{k=1} ^\infty {\left( \mu_k +\frac{\tau_k}{\lambda_k^2} \right)\langle .,e_k\rangle e_k} 
\label{sigmax}
\end{equation}
are trace class operators.

\medskip\noindent Thus, in view of  Proposition \ref{Mandelbaum}, we have
\begin{equation}\label{cond-H}
E[y|x] =y_0+Q_v\left[\Sigma_u+Q_v\right]^{-1}(x-y_0),
\end{equation}
provided that the operator  
\[
\begin{split}
T&:= \Sigma_{XY} \Sigma _{X} ^{-\frac{1}{2}}\\
& = \sum_{k=1} ^\infty {\frac{\tau _k}{\lambda_k^2} \left( \mu_k +\frac{\tau_k}{\lambda_k^2} \right)^{-\frac{1}{2}} \langle .,e_k\rangle e_k}
\end{split}
\]
 is a Hilbert-Schmidt, i.e.  
\begin{equation}
\label{Thilbertshmidt*}
\begin{array}{lll}
\left\| T \right\|^2_2 = \sum_{k=1}^{\infty} {\left\|Te_k\right\|^2 }= \sum_{k=1}^{\infty}{\left(\frac{\tau_k}{\lambda_k^2}\right)^2 \left( \mu_k +\frac{\tau_k}{\lambda_k^2} \right)^{-1}} \\ \,\,\,\qquad= \sum_{k=1}^{\infty}{\frac{\tau _k}{\lambda_k^2} \left(\frac{\lambda_k^2 \mu_k}{\tau_k}+1 \right)^{-1}} < \infty.
\end{array}
\end{equation}

\medskip\noindent In this case, 
\begin{equation*}
\begin{split}
E[y|x]& =y_0+\sum_{k=1}^{\infty} \langle x-y_0,e_k\rangle \langle \Sigma_{XY} \Sigma _{X}^{-1} e_k ,e_k \rangle e_k\\ & = y_0+\sum_{k=1}^{\infty}  \langle x-y_0,e_k \rangle \frac{\tau_k}{\lambda_k^2} \left( \mu_k +\frac{\tau_k}{\lambda_k^2}\right)^{-1}  e_k
\end{split}
\end{equation*}
or
\begin{equation}\label{cond-H-1}
E[y|x]= y_0+ \sum_{k=1}^{\infty} \langle x-y_0,e_k \rangle \left( 1 +\frac{\lambda_k^2 \mu_k }{\tau_k} \right)^{-1} e_k .
\end{equation}

\medskip\noindent The following theorem is the main result of the paper. It is the infinite dimensional extension of Propositions 2.2 and 3.4 in Dermoune {\it et al.} (2009). 
\begin{theorem}\label{optimalb}
Let Assumptions (1) to (5)  hold, and that
\begin{equation}
\label{Thilbertshmidt}
\left\| T \right\|^2_2=\sum_{k=1}^{\infty}{\frac{\tau_k}{\lambda_k^2} \left(\frac{\lambda_k^2 \mu_k}{\tau_k}+1 \right)^{-1}} < \infty,
\end{equation}
then, for all $x\in H_1$, the smoothing operator (which is linear, compact and injective)
\begin{equation}\label{B-hat}
\hat Bh:=(AA^*)^{-1}A\Sigma_u A^*\Sigma_v^{-1}h=\sum_{k=1}^\infty { \frac{\mu_k}{\tau_k} \langle h,d_k \rangle d_k},\qquad h\in H_2,
\end{equation}
where, the sum converges in the operator norm, is the unique operator which satisfies
\[ \hat B = \arg\min_B\left\| y(B ,x)- E[y|x] \right\|_{H_1}, \]
where the minimum is taken with respect to all linear bounded operators which satisfy the positivity condition (\ref{B}).

\medskip\noindent Furthermore, we have
\begin{equation}\label{funda}
y(\hat B ,x)- E[y|x]=(I_{H_1}-\Pi)(x-E[x]),
\end{equation}
and its covariance operator is
\begin{equation}\label{cond-cov}
\mbox{cov } (y(\hat B ,x)- E[y|x])=(I_{H_1}-\Pi)\Sigma_u.
\end{equation}
In particular, 
\begin{equation}
E\left(\left\| y(\hat B ,x)-E[y|x] \right\|^2_{H_1}\right)=\mbox{trace}\left((I_{H_1}-\Pi)\Sigma_u\right).
\end{equation}
\end{theorem}

\begin{proof} Denote $\bar\Pi:=I_{H_1}-\Pi$. For any linear and bounded operator $B$ which satisfies (\ref{B}), we have
$$
y(B ,x)- E[y|x]=\Lambda_1+\Lambda_2,
$$
where,
$$
\Lambda_1:=\bar\Pi x-y_0,\qquad \Lambda_2:=\left(I_{H_1}+A^*BA\right)^{-1}x-\bar\Pi x-Q_v\left[\Sigma_u+Q_v\right]^{-1}(x-y_0).
$$
Therefore,
$$
\left\| y(B ,x)- E[y|x] \right\|^2_{H_1}=\left\| \Lambda_1 \right\|^2_{H_1}+\left\| \Lambda_2\right\|^2_{H_1}+2\langle \Lambda_1,\Lambda_2\rangle.
$$
Since, $\bar\Pi x-y_0\in \mbox{Ker}(A)$ and $Q_v\bar\Pi=0$, it follows that 
$$
\langle \bar\Pi x-y_0, Q_v\left[\Sigma_u+Q_v\right]^{-1}(x-y_0)\rangle =0.
$$
Hence, noting that $y_0=(I_{H_1}-\Pi)y:=\bar\Pi y$, we get
$$
\langle \Lambda_1,\Lambda_2\rangle=\langle \bar\Pi (x-y),\left(\left(I_{H_1}+A^*BA\right)^{-1}-\bar\Pi\right)x\rangle.
$$
But, expressing the operator $\left(I_{H_1}+A^*BA\right)^{-1}$ as a power series of operators, in view of the fact that $\Pi A^*=A^*$ i.e. $\bar\Pi A^*=0$, we have
\begin{equation}\begin{array}{lll}
\bar\Pi\left(I_{H_1}+A^*BA\right)^{-1}&=\bar\Pi\left (I_{H_1}-A^*BA+(A^*BA)^2-\cdots\right)\\ &=
\bar\Pi-\bar\Pi A^*BA+\bar\Pi A^*BAA^*BA-\cdots\\  &=\bar\Pi.
\end{array}
\end{equation}
Hence, 
\begin{equation*}
\langle \Lambda_1,\Lambda_2\rangle=\langle \bar\Pi (x-y),\left(\left(I_{H_1}+A^*BA\right)^{-1}-\bar\Pi\right)x\rangle =0.
\end{equation*}
This yields
$$
\left\|y(B ,x)-E[y|x] \right\|^2_{H_1}=\left\| \Lambda_1 \right\|^2_{H_1}+\left\| \Lambda_2\right\|^2_{H_1}.
$$
Hence, 
\begin{equation}\label{ineq}
\inf_{B}\left\| y(B ,x)-E[y|x] \right\|^2_{H_1}\ge\left\| \bar\Pi x-y_0\right\|^2_{H_1}.
\end{equation}
It remains to show that there is a linear and bounded operator $\hat B$ which satisfies(\ref{B})  and  for which the lower bound in (\ref{ineq}) is attained. Indeed, such an operator satisfies
$$
y(\hat B ,x)- E[y|x]=\bar\Pi x-y_0,\qquad x\in H_1, \,\,y_0\in \mbox{Ker}(A).
$$
This equality holds if and only if
\begin{equation}\label{y0-1}
Q_v\left[\Sigma_u+Q_v\right]^{-1}y_0=0
\end{equation}
and
\begin{equation}\label{y0-2}
\left(I_{H_1}+A^*\hat BA\right)^{-1}=I_{H_1}-\Pi+Q_v\left[\Sigma_u+Q_v\right]^{-1}.
\end{equation}
In fact, (\ref{y0-2}) implies (\ref{y0-1}). Indeed,
$$
\begin{array}{lll}
Q_v\left[\Sigma_u+Q_v\right]^{-1}y_0&=\left(\left(I_{H_1}+A^*BA\right)^{-1}-(I_{H_1}-\Pi)\right)y_0\\ &
 =\sum_{n=0}^{\infty}(-1)^n\left(A^*BA\right)^ny_0-y_0+\Pi y_0\\& =-y_0+ \Pi y_0+y_0+\sum_{n=1}^{\infty}(-1)^n\left(A^*BA\right)^{n-1}A^*BAy_0 \\ &=0,
\end{array}
$$
since, $y_0\in\mbox{Ker}(A)$ i.e. $Ay_0=0$ and a fortiori $\Pi y_0=0$.

\medskip\noindent It remains to find $\hat B$ which solves (\ref{y0-2}). Indeed, multiplying (\ref{y0-2}) with $I_{H_1}+A^*\hat BA$ and then with $\Sigma_u+Q_v$, we get 
$$
\Pi\Sigma_u=A^*\hat B\Sigma_v\left(AA^*\right)^{-1}A.
$$
Multiplying both sides of this equality with $A^*$ and then with $(AA^*)^{-1} A$, it holds that
\begin{equation}
\hat B=(AA^*)^{-1}A\Sigma_u A^*\Sigma_v^{-1}.
\end{equation}
Now, inserting this value of $\hat B$ in (\ref{y0-2}), we necessarily get $\Pi\Sigma_u=\Pi\Sigma_u\Pi$, which holds if and only if $\Pi\Sigma_u=\Sigma_u\Pi$ which is nothing but Assumption 4. That is, Equation(\ref{y0-2}) is solvable if and only if  Assumption 4 holds. Relation (\ref{cond-cov}) is immediate. This finishes the proof.
\end{proof}

\begin{remark}\label{fundamental}

\medskip
\item[1.] Formula (\ref{funda}) displays an interesting representation of the "optimal" filter i.e. the conditional expectation of the signal $y$ given the functional data $x$, in terms of the solution of the purely deterministic inverse problem (\ref{mini}). Namely,
\begin{equation}\label{opt-filter}
 E[y|x]=\left(I_{H_1}+\Sigma_u A^*\Sigma_v^{-1}A\right)^{-1}x-(I_{H_1}-\Pi)(x-E[x]).
\end{equation}

\medskip
\item[2.] The set of parameters $\{\lambda_n, \mu_n,\tau_n,\,\, n\ge 1  \}$ for which the operators $\Sigma_u, Q_v$ are trace class, $T$ is Hilbert-Schmidt and the series (\ref{B-hat}) converges in the operator norm, is not empty. Indeed, for $\lambda_n=n^{-\alpha},\, \mu_{n}=n^{-\beta},\, \tau_n=n^{-\gamma}$, where
$\alpha\ge 2,\, \beta\ge 2\alpha+4,\, \gamma\ge 2\alpha+2$, these operators satisfy the required properties. 

\end{remark}
\subsection*{Example.}
We apply Theorem 4 to the well known backward heat conduction problem (see for instance \cite{nair}). Let $u(s,t)$ represent the temperature at a point $s$ on a thin wire of length $\pi$ at time $t$. Assuming the wire end points are kept at zero temperature. Then $u( \cdot, \cdot)$ satisfies the heat equation 
$$
\frac{\partial u}{\partial t} = \frac{\partial^2 u}{\partial ^2 s} 
$$
with $u(0,t)=0 = u(\pi ,t)$. 

\noindent The problem is the following: knowing the temperature $v(s) := u(s, \tau ), \,\, 0< s \leq \pi,$ at time $t= \tau$, determine the temperature $ y(s): = u(s, t_0)$ at time $ t= t_0 < \tau $.

\medskip\noindent The problem of finding $y$ from the knowledge of $v$ is equivalent to solving the compact operator equation 
$$ Ky= v$$ 
where $ K : L^2 [0, \pi ] \rightarrow  L^2 [0, \pi ]$ is given by (see \cite{nair} fore further details)
$$
(K_{t_0} y)(s)= \sum _ {n=1} ^\infty {e^{-n^2 (\tau - t_0)} \langle y(s) , e_n (s) \rangle e_n(s) },
$$
where, $e_n(s)= \sqrt{ \frac{2}{\pi}} \sin(ns)$ form an orthogonal basis of $L^2 [0,\pi]$. 

\noindent Therefore, the equation $Ky = v$ has a solution of the form 
$$
y(s) = y_0 +\sum _ {n=1} ^\infty {e^{n^2 (\tau - t_0)} \langle v(s) , e_n (s) \rangle e_n(s) }
$$
where, $y_0 \in \ker ( K)$.

\medskip\noindent Now consider the Hodrick-Prescott filter associated the operator $K$. 
Under the assumption that $u$ and $v$ are independent Gaussian random variables with zero means and covariance operators of trace class of the form
$$
\Sigma_u h(s) = \sum _{n=1}^\infty {\sigma_n ^u  \langle h(s) , u_n(s) \rangle u_n(s)}, 
$$
and
$$
\Sigma_v h(s) = \sum _{n=1}^\infty {\sigma_n ^v  \langle h(s) , u_n(s) \rangle u_n(s)},
$$
respectively, where the sums converge in the operator norm. By Theorem \ref{optimalb}, the optimal smooth operator $B$ given by (\ref{B-hat})  reads 
$$
(\hat{B} h)(s) =  \sum _{n=1}^\infty {\frac{ \sigma_n ^u }{\sigma _n ^v}  \langle h(s) , e_n(s) \rangle e_n(s)}, \qquad h  \in L^2[0,\pi].
$$
The corresponding optimal signal given by (\ref{mini-2}) is
$$
y(\hat B, x) (s)= \sum _{n=1}^\infty {\left(1+\frac{\sigma_n ^u }{\sigma _n ^v}e^{ -2n^2 (\tau -t_0)}\right)^{-1}\langle x(s) ,e_n(s) \rangle e_n(s)}. 
$$

\section{Extension to non-trace class covariance operators} 

In this section we will extend the characterization of the optimal smoothing parameter to the case where the covariance operators of $u$ and $v$ are not necessarily trace class operators. More precisely, we assume that $u \sim N(0,\Sigma _u)$ and $v \sim N(0,\Sigma _v)$ where $\Sigma _u$ and $ \Sigma _v$ are self-adjoint positive-definite and bounded  but not trace class operators on $H_1$ and $H_2$, respectively. This extension includes the important case where $u$ and $v$ are white noise with covariance operators of the form $\Sigma_u=\sigma_u^2I_{H_1}$ and $\Sigma_v=\sigma_v^2I_{H_2}$, respectively, for some constants $\sigma_u$ and $\sigma_v$.

\medskip\noindent  Following Rozanov (1968) (see also Lehtinen {\it et al.} (1989)),  we can look at these Gaussian variables as generalized random variables on an appropriate Hilbert scale (or nuclear countable Hilbert space), where the covariance operators can be maximally extended to self-adjoint positive-definite, bounded and trace class operators on an appropriate domain. 

 \medskip\noindent  We first construct the Hilbert scale  appropriate to our setting. This is performed using the compact operator $A$ as follows (see Engle {\it et al.} (1996) for further details).

\medskip\noindent 
Set $K_1:= (A^*A)^{-1} $. It is densely defined, unbounded, self-adjoint positive-definite, injective, linear operator in the Hilbert space $H_1$. Furthermore,  
\begin{equation}
\label{K1}
K_1 h = \sum_{k=1}^{\infty}\lambda _k ^{-2} \langle h,e_k \rangle e_k,\qquad h\in H_1.  
\end{equation}
We can define the fractional power of the operator $K_1$ by
\begin{equation}
\label{s-K1}
K^s_1 h = \sum_{k=1}^{\infty}\lambda_k ^{-2s} \langle h,e_k \rangle e_k,\quad h\in H_1,\quad s\ge 0,
\end{equation}
and define its domain by
\begin{equation}
\mathcal{D}(K_1^{s}):=\left\{h\in H_1; \quad \sum_{k=1}^{\infty}\lambda _k ^{-4s} |\langle h,e_k \rangle|^2<\infty \right\}.
\end{equation}

\medskip\noindent Let  $ \mathcal{M} $ be the set of all elements $x$ for which all the powers of $K_1$ are defined i.e 
$$ \mathcal{M} := \bigcap \limits _{n=0} ^\infty \mathcal {D} (K_1^{n}).$$
For $s \geq 0$, let $H_1^s$ be the completion of $ \mathcal{M} $  with respect to the Hilbert space norm  induced by the inner product 
\begin{equation}
 \langle x,y\rangle _{H_1^s} :=  \langle K_1^s x,K_1^s y\rangle_ {H_1}, \quad x,y \in \mathcal{M},
\end{equation}
and let  $H_1 ^{-s} := (H_1^s) ^*$ denotes the dual of $H_1^s$ equipped with the following inner product: 
\begin{equation}
 \langle x,y\rangle_{H_1 ^{-s}} :=  \langle K_1^{-s}x,K_1^{-s} y\rangle_{H_1}, \quad x,y \in \mathcal{M}.
\end{equation}
Then,  $(H_1^t)_{t\in \mathbb{R}}$  is the Hilbert scale induced by the operator $K_1$. In particular we have  the following chain of dense continuous embeddings 
\[N \subset \ldots \subset H_1 ^n \subset \ldots H_1^2 \subset H_1 ^1 \subset H_1 \subset H_1^{-1} \subset H_1^{-2} \subset \ldots \subset H_1^{-n}\subset \ldots \subset N^*,\]
where,    $N= \bigcap \limits _{n=0}^\infty  H_1^n $, equipped with the projective limit topology, and $ N^* = \bigcup \limits _{n=0}^ \infty H_1 ^{-n}$,   equipped with the weak topology.

\medskip\noindent The operator $K_2:= (AA^*)^{-1} $ in $H_2$ has the same properties as $K_1$ with 
\begin{equation}
\label{K2}
K_2 h = \sum_{k=1} ^\infty {\lambda_k^{-2} \langle h,d_k\rangle d_k},\quad h\in H_2.
\end{equation} 
Repeating the same procedure as before we get the following chain of dense continuous embeddings 
\[M \subset \ldots \subset H_2 ^n \subset \ldots H_2^2 \subset H_2 ^1 \subset H_2 \subset H_2^{-1} \subset H_2^{-2} \subset \ldots \subset H_2^{-n} \subset \ldots\subset M^*,\]  
where, the norm in $H_2^n$ is given by  $\left\| h \right\| _{H_2^{n}} = \left\| K_2 ^{n} h \right\|_{H_2} $,  $\quad h \in H_2^n.$\\

\medskip\noindent Noting that
\[ H_1^{-n} = \mbox{Im} \left( (A^*A)^n \right) = (A^*A)^n (H_1), \]
 and  
\[ H_2^{-n} = \mbox{Im} \left( (AA^*)^n \right) = (AA^*)^n (H_2). \]
Moreover,  $ \ker (A) = \ker \left( (A^*A)^n \right)$ and $\ker(A^*)= \ker \left( (AA^*)^n \right)$, it follows that the operator $A$ extends to a continuous operator from $H_1^{-n}$ into $ H_2^{-n}$, and the operators  $A^*A$ and  $AA^*$ extend as well to a continuous operator onto $H_1^{-n}$ and $H_2^{-n}$ respectively.

\medskip\noindent 
The flexibility offered by the Hilbert scale allows us to extend the HP filter to  the larger Hilbert spaces $H_1^{-n}$ and $H_2^{-n}$, where $n$ is chosen so that  the second moments $ E[\|x\|^2_{H_1 ^{-n}}], E[\|y\|^2_{H_1 ^{-n}}]$, $E[\|u\|^2_{H_1 ^{-n}}]$ and $E[\|v\|^2_{H_2^{-n}}]$
of the Gaussian random variables $x, y, u$ and $v$  in $H_1^{-n}$ and $H_2^{-n}$ respectively, are finite. This amounts to make their respective covariance operators 
\begin{equation}
\label{tildeuv}
\tilde{\Sigma }_u = (A^*A)^n \Sigma _u (A^*A)^n, \quad \tilde{\Sigma }_v = (AA^*)^n \Sigma _v (AA^*)^n
\end{equation}
and 
\begin{equation}
\tilde\Sigma= \left( \begin{array}{ccc}
\widetilde{\Sigma} _u + \widetilde{Q}_v  && \widetilde{Q}_v \\ 
\widetilde{Q}_v &&  \widetilde{Q}_v 
\end{array} 
\right), 
\end{equation}
where,
\begin{equation}\label{Qv}
\widetilde{Q}_v: =  (A^*A)^n Q_v (A^*A)^n = A^*(AA^*)^{-1}\widetilde{\Sigma}_v(AA^*)^{-1}A,
\end{equation}
trace class.

\medskip We make the following assumption:

\medskip{\bf Assumption 6.} There is $n_0 >0 $ such that  for all $n\geq n_0$ we have 
\[ \sum _{k=1}^\infty {\lambda_k ^{4n-2} \mu_k } < \infty \qquad\text{and}\qquad  \sum _{k=1}^\infty {\lambda_k ^{4n} \tau_k } < \infty. \]
\medskip\noindent  We note that this assumption implies that 
$\sum _{k=1}^\infty {\lambda_k ^{4n} \mu_k } < \infty,$ for $n\ge n_0$.

\medskip\noindent Under Assumption 6, the covariance operators $\tilde\Sigma_u, \tilde\Sigma$ and $ \tilde\Sigma_v$  are trace class on the Hilbert spaces $H_1^{-n}$ and  $H_2^{-n}$, respectively. Moreover, noting that since $y_0 \in \ker ({A})= \ker \left( (A^*A)^n \right)$ it follows that  $\left\| y_0 \right\|_ {H_1^{-n}} = \left\| (A^*A)^n y_0 \right\|_ {H_1} =0$. Hence, the $H_1^{-n}\times H_1^{-n}$-valued random vector $(x,y)$ has mean $(E[x],E[y])=(0, 0)$.

\medskip\noindent  Summing up, by Assumption 6, for $n\ge n_0$, the vector $(x,y)$ is an $H_1^{-n}\times H_1^{-n}$-valued Gaussian vector with mean $(0,0)$ and covariance operator $\tilde\Sigma$.  Thus, in view of Proposition \ref{Mandelbaum}, we have 
\begin{equation} 
E[y|x] =\widetilde{Q}_v\left[\widetilde{\Sigma}_u+\widetilde{Q}_v\right]^{-1}x, \qquad \mbox{a.s. in } H_1^{-n}.
\label{expn}
\end{equation}
provided that the operator 
\begin{equation}\label{HS-N}
\tilde T := \widetilde{\Sigma}_{XY} \widetilde{\Sigma} _X^{-\frac{1}{2}}= \sum_{k=1} ^\infty {\tau_k \lambda_k ^{2(2n-1)} \left( \mu_k \lambda_k ^{4n}  +\tau_k \lambda_k ^{2(2n-1)} \right)^{-\frac{1}{2}} \langle \cdot, e_k \rangle e_k}
\end{equation}
is Hilbert-Schmidt, i.e.
\begin{equation}\label{ThilbertshmidtN}
 \left\| \tilde T \right\| ^2 = \sum_{k=1} ^\infty {\tau_k \lambda_k ^{2(2n-1)} \left( \frac{\mu_k}{\tau_k} \lambda_k ^2  + 1 \right)^{-1}} < \infty.
\end{equation}
But,  in view of Assumption 6, this series is finite.

\medskip\noindent
The deterministic optimal signal associated with $x$ in $H_1^{-n}, n \geq n_0$, is given by the minimizer of the following functional 
\begin{equation}\label{HP-HN-trend}
J_B(y) =  \left\| x-y \right\|_{H_1^{-n}}^2 +  \langle Ay, BAy\rangle_{H_2^{-n}} ,
\end{equation} 
with a linear bounded operator $B: H_2^{-n}\longrightarrow H_2^{-n}$ such that $\langle Ah, BAh\rangle_{H_2^{-n}} \ge0 $ for all $h \in H_1^{-n}$.  
This minimizer is given by the formula (cf. Proposition \ref{mainprop})  
\begin{equation}
y(B,x) = ( I_{H_1^{-n}} + A^* BA ) ^{-1} x.
\label{opty}
\end{equation}
An explicit expression of the optimal smoothing parameter $\hat{B} $ is given in the following 

\begin{theorem}\label{optB}
Let Assumption 6 hold. Then, the operator 
\begin{equation}\label{hat-BN}
\hat Bh:=(AA^*)^{-1}A\tilde\Sigma_u A^*\tilde\Sigma_v^{-1}h, \quad h\in H_2^{-n},
\end{equation}
 is the unique optimal  smoothing operator associated with the HP filter associated with $H_1^{-n}$-valued data $x$.
\end{theorem}

\begin{proof}  The minimizer $\hat B$ of   $\left\| y(B ,x)- E[y|x] \right\|_{H_1^{-n}}$ is determined by imposing 
\begin{equation}
\label{bhatN}
 E[y|x] =  \widetilde{Q}_v\left[\widetilde{\Sigma}_u+\widetilde{Q}_v\right]^{-1}x =  ( I_{H_1^{-n}} + A^* BA ) ^{-1} x,  \quad  x \in H_1^{-n}.
\end{equation}
 
\medskip\noindent Multiplying (\ref{bhatN}) with $I_{H_1}+A^*\hat BA$ and then with $\widetilde{\Sigma}_u+\widetilde{Q}_v$, we get 
$$
\widetilde{\Sigma}_u = A^* \hat{B} \widetilde{\Sigma}_v (AA^*)^{-1}A.
$$
Furthermore, multiplying both sides of this equality with $A^*$ and then with $(AA^*)^{-1}A$, we finally obtain
\begin{equation}\label{B-hat-2}
\hat B=(AA^*)^{-1}A\widetilde{\Sigma}_u A^*\widetilde{\Sigma}_v^{-1}.
\end{equation}
This finishes the proof.
\end{proof}

\subsection{The white noise case-  Optimality of the noise-to-signal ratio}
In this section we  apply Theorem \ref{optB} to the case where  $u$ and $v$ are white noise. We will show that the optimal smoothing operator $\hat B$ given by (\ref{B-hat-2}) reduces to  the noise-to-signal ratio in the same fashion as for the classical HP filter. Indeed, assume  $u$ and $v$ are independent Gaussian random variables with zero means and covariance operators  $\Sigma_u = \sigma_u I_{H_1 }$ and $\Sigma_v = \sigma_v I_{H_2 }$, where $I_{H_1 }$ and $I_{H_2 }$ denotes the $H_1$ and $H_2$ identity operators, respectively and $\sigma_u$ and $\sigma_v$ are constant scalars. Assumption 4, reduces to

\medskip{\bf Assumption 7.} There is an $n_0>0$ such that  $\sum _{k=1} ^ \infty {\lambda _k^{2(2n-1)} } < \infty $ for all $n \ge n_0$.\\

\medskip\noindent Under this assumption, the  associated covariance operators 
\begin{equation*}
\tilde{\Sigma }_u= \sigma_u \sum_{k=1}^\infty {\lambda_k ^{4n} \langle \cdot,e_k\rangle e_k},\quad
\tilde{\Sigma }_v = \sigma_v \sum_{k=1}^\infty {\lambda_k ^{4n} \langle \cdot,d_k\rangle d_k},\quad \widetilde{Q}_v=  \sigma_v \sum_{k=1}^\infty {\lambda_k ^{2(2n-1)} \langle \cdot,e_k\rangle e_k}.
\end{equation*}
 are all trace class operators. Moreover, the norm (\ref{ThilbertshmidtN}) of the operator $\tilde T$  which reads
\begin{equation}
 \left\| \tilde T \right\| ^2 = \sum_{k=1} ^\infty {\sigma_v \lambda_k ^{2(2n-1)} \left( \frac{\sigma_u}{\sigma_v} \lambda_k ^2  + 1 \right)^{-1}},
 \label{hsn}
\end{equation}
 is finite, making $\tilde T$ Hilbert-Schmidt.  Hence, the expression (\ref{B-hat-2}) giving the optimal smoothing operator $\hat B$ reduces to 

 \begin{equation}
\hat B = (AA^*)^{-1}A\Sigma_u A^*\Sigma_v^{-1}= \frac {\sigma_u}{\sigma_v} \sum_{k=1}^\infty { \langle \cdot,d_k\rangle d_k} = \frac {\sigma_u}{\sigma_v} I_{H_2^{-n}},
\end{equation}

i.e. $\hat B$ is the noise-to-signal ratio.


\end{document}